\documentclass[runningheads,a4paper,envcountsame]{llncs}

\usepackage{amsfonts}
\usepackage{amssymb}
\usepackage{mathtools}
\usepackage[all,cmtip]{xy}
\usepackage[bookmarks=false]{hyperref}
\usepackage{xcolor}
\usepackage{todonotes}
\usepackage{url}
\newtheorem{defini}{Definition}

\newtheorem{prop}{Proposition}

\renewcommand{\phi}{\varphi}

\usepackage{todonotes}
\usepackage{tikz}

\title{Effective Reducibility for Statements of Arbitrary Quantifier Complexity with Ordinal Turing Machines}
\author{Merlin Carl}
\date{August 2025}
\institute{Institut f\"ur Mathematik, Europa-Universit\"at Flensburg}

\begin{document}

\maketitle





\begin{abstract}
This paper is an extended version of our work in \cite{Ca2025}. 
We extend the concept of effective reducibility between statements of set theory with ordinal Turing machines (OTMs) explored in \cite{Ca2018} for $\Pi_{2}$-statements to statements in prenex normal form of arbitrary quantifier complexity and use this to compare various fundamental set-theoretical principles, including the power set axiom, the separation scheme, the collection and replacement schemes and several principles related to the notion of cardinality, with respect to effective reducibility. This notion of reducibility is both different from (i.e., strictly weaker than) classical truth and from the OTM-realizability of the corresponding implications. Along the way, we obtain a computational characterization of HOD as the class of sets that are OTM-computable relative to every effectivizer of $\Sigma_{2}$-separation. We also consider an associated variant or Weihrauch reducibility. 
\end{abstract}

\section{Introduction}

A $\Pi_{2}$-statement $\forall{x}\exists{y}\phi(x,y)$ (with $\phi$ a $\Delta_{0}$-formula) can be regarded as \textit{effective} if there is an effective method for constructing a corresponding $y$ for each given $x$.  
In effective reducibility, one then asks whether, given access to any method effectivizing a statement $\psi$, one can effectivize $\phi$ relative to that method. Usually, ``effective method'' is interpreted as ``Turing program'' (see, e.g., Hirschfeldt, \cite{Hirschfeldt}, pp. 22f). 

In \cite{Ca2016}, we introduced a variant of effective reducibility based on Koepke's ordinal Turing machines (OTMs) rather than classical Turing machines; in \cite{Ca2018}, we conducted an investigation concerning the OTM-reducibility of various formulations of the axiom of choice. However, as sketched in section $6$ of \cite{Ca2018}, the ideas behind effectivity and reducibility can be extended to arbitrary set-theoretic statements, provided they are given in prenex normal form. We also briefly mentioned that, in this setting, the separation scheme was reducible to the replacement scheme. 

In \cite{Ca2025}, we began a study of this extension of effective OTM-reducibility to statements of arbitrary quantifier complexity. The present work is an expanded version of  \cite{Ca2025} with several new results. In particular, the principles Card, DecCard, OrdCard and PowerCard (to be defined below) and all results concerning these principles are original contributions of this paper, as is Theorem \ref{pot sep hod}, which yields a complete characterization of models of set theory in which the power set axiom is reducible to $\Sigma_{2}$-separation. Most of the other results are from \cite{Ca2025}.

Another way to ``effectivize'' set-theoretical statements with OTMs is to consider notions of \textit{realizability} based on OTMs; this path was pursued in \cite{CGP}. 
These approaches agree in the sense that, for formulas in prenex normal form, OTM-effectivity is equivalent to uniform OTM-realizability (\cite{CGP}, Definition $18$) with a single ordinal parameter. 
However, the OTM-effective reducibility of $\phi$ to $\psi$ differs considerably from the OTM-realizability of $\psi\rightarrow\phi$: While an OTM-realizer must be able to convert OTM-\textit{realizers} for $\psi$ into OTM-realizers for $\phi$, an OTM-effective \textit{reduction} must achieve the analogue for \textit{arbitary} effectivizers, which, from the point of view of OTMs, may be entirely  ineffective. As we will see below, OTM-effective reducibility is indeed strictly stronger than the OTM-realizability of the corresponding implication, for formulas in prenex normal form.

In particular, we will investigate interdependencies between the axioms of ZFC in terms of effective reducibility, showing that an effectivizer for the separation scheme makes all heriditarily ordinally definable sets computable (Theorem \ref{full sep and hod}), that the power set axiom is effectively reducible to the $\Sigma_{3}$-replacement scheme (Theorem \ref{pot}(4)) and that replacement is reducible to the conjunction of the separation scheme and the power set axiom (Lemma \ref{sep and rep}(1)), while a number of other reductions, such as the reducibility of $\Sigma_{3}$-replacement to separation alone (Lemma \ref{sep and rep}(4)), or the the reducibility of the power set axiom to the statement that each cardinal has a successor (Theorem \ref{pot}(3)), are independent of ZFC. In such cases, one can ask whether the reducibility statement is equivalent to some known axiom extending ZFC; an example of this is Theorem \ref{pot sep hod}.

\section{Basic Definition}

Our notation is mostly standard. We use $\mathfrak{P}(x)$ for the power set of $x$, $\text{Card}$ for the class of cardinals, $\text{card}^{M}(x)$ for the cardinality of $x$ in the model $M$ (where the superscript $M$ is dropped when the universe is clear from the context), $\text{tc}(x)$ for the transitive closure of $x$ and $\kappa^{+}$ for the next cardinal after $\kappa$, where $\kappa$ is an ordinal. For ordinals $x$ and $y$, $p(x,y)$ is Cantor's ordinal pairing function; if $z$ is a code for a pair, we write $(z)_{0}$ and $(z)_{1}$ for its components. Arbitrary sets are encoded as sets of ordinals in the standard way (cf., e.g., \cite{CarlBuch}, Definition 2.3.18): For a set $x$, fix a bijection $f:\alpha\rightarrow\text{tc}(\{x\})$ (where $\alpha\in\text{On}$) such that $f(0)=x$, then code $x$ as $\{p(\iota,\xi):f(\iota)\in f(\xi)\}$. For an (OTM-)program $P$, we write $P^{F}(x,\rho)\downarrow=y$ to indicate that the program $P$, using the extra class function $F$ (see below) in the input $x$ and the parameter $\rho$, will halt with output $y$. Occasionally, the $\downarrow$ will be dropped.

We assume familiarity with Koepke's ordinal Turing machines (OTMs) and refer to \cite{Koepke} and \cite{CarlBuch} for details. In this paper, OTM-computations always allow for ordinal parameters, i.e., at the start of the computation, some ordinal on the scratch tape may be marked. (Many of our results, however, do not depend on the presence of parameters.) We will refer to a pair $(P,\rho)$ consisting of an OTM-program $P$ and an ordinal parameter $\rho$ as a \textit{parameter-program}. 
 The following definition is then a variant of a proposal in \cite{Ca2018}, section $6$:

\begin{defini}{\label{effectivizer}}
Let $\phi:\Leftrightarrow \forall{x_{0}}\exists{y_{0}}...\forall{x_{n}}\exists{y_{n}}\psi(x_{0},y_{0},...,x_{n},y_{n})$ be an $\in$-formula in prenex normal form, where $\psi$ is quantifier-free. An \textit{effectivizer} of $\phi$ is a class function $F:V\rightarrow V$ such that, for all $x_{0},...,x_{n}$, we have\footnote{Note the double brackets in $F((x_{0}))$, which indicate that $F$ is applied to the sequence consisting of the single element $x_{0}$, rather than to $x_{0}$ itself.} $$\psi(x_{0},F((x_{0})),x_{1},F((x_{0},x_{1})),...,x_{n},F((x_{0},...,x_{n}))).$$ 


An \textit{encoding} $\hat{F}$ of $F$ is a class function $F:\mathfrak{P}(\text{On})\rightarrow\mathfrak{P}(\text{On})$ such that, for all sets $x$, whenever $c_{x}\subseteq\text{On}$ codes $x$, then $\hat{F}(c_{x})$ codes $F(x)$. 

If $\hat{F}$ is OTM-computable by a parameter-program $(P,\rho)$, we call $(P,\rho)$ an \textit{OTM-effectivizer} of $\phi$ and say that $\phi$ is \textit{OTM-effective}. 

Thus, a statement is effective if some effectivizer of it has an OTM-computable encoding. 

This definition extends in a straightforward way to \textit{schemes} of statements: If $\Phi:=\{\phi_{i}:i\in\omega\}$ is a set of $\in$-statements,\footnote{Generalizations to uncountable schemes are easily possible, but will not be considered in this work.} then an effectivizer for $\Phi$ is a class function $F:\omega\times V\rightarrow V$ such that, for all $i\in\omega$, the function $x\mapsto F(i,x)$ is an effectivizer for $\phi_{i}$. As usual, we will write $\lceil\psi\rceil$ for that element $i\in\omega$ such that $\psi=\phi_{i}$. If $\phi$ is a formula and $p$ is a set parameter, then $\lceil\phi(p)\rceil$ abbreviates $(\lceil\phi\rceil,p)$.

\end{defini}


We will also consider cases where only the first few, rather than all, existential quantifiers in a statement are witnessed. For example, a statement such as ``Every ordinal has a cardinality'' has the form $\forall{\alpha}\exists{\beta}\exists{f}\forall{\gamma<\beta}\forall{g}(f:\beta\xrightarrow[]{\text{1:1}}\alpha\wedge\neg g:\gamma\xrightarrow[\text{}]{\text{1:1}}\beta)$. Clearly, finding the ``right'' ordinal is quite different from also offering a corresponding bijection. 
For this purpose, however, the above needs to be refined slightly. While it is always possible to contract several existential quantifiers into one by quantifying over a single tuple, distinguishing between different quantifiers in the same block makes a difference at this point.  Thus, we define:

\begin{defini}
    For a statement $\phi$ of the form $$\forall{x_{1}}\exists{y_{11},...y_{1k_{1}}}...\forall{x_{n}}\exists{y_{n1},...,y_{nk_{n}}}\psi(x_{1},y_{11},...,y_{1k_{1}},...,x_{n},y_{n1},...,y_{nk_{n}})$$ 
    and $k\leq k_{1}+...+k_{n}$, we say that $F:V\rightarrow V$ is a $k$-effectivizer for $\phi$ if and only if the following holds: For all $j\leq k$, $F$ satisfies the definition of an effectivizer for the first $k$ existentially quantified variables in $\phi$. 

    The notions of encoded $k$-effectivizer, OTM-$k$-effectivizer, OTM-$k$-effectivity and the extension to schemes work by the obvious analogy to the last statement.

\end{defini}

Thus, intuitively, a $k$-effectivizer of a statement $\phi$ in prenex normal form finds witnesses for the first $k$ existential quantifiers in $\phi$. 
Note that, as a consequence of this, the concept of a $k$-effectivizer of $\phi$ depends on the precise way $\phi$ is stated, and in particular on the ordering of the existential quantifiers in a block of existential quantifiers. In order to avoid confusion, we will below state explicitly what a $k$-effectivizer of a certain statement is intended to achieve in cases where this is not immediately clear.

\begin{remark}
    Below, we will slightly abuse our terminology by using the term ``effectivizers'' to also refer to encodings of effectivizers.
\end{remark}

\begin{remark}
    The same definitions work for \textit{game formulas}, 
    i.e., infinite formulas of the form $\forall{x_{0}}\exists{y_{0}}\forall{x_{1}}\exists{y_{1}}...\psi(x_{0},y_{0},...)$ in which an infinitary quantifier-free formula $\psi$ is preceded by a series of quantifiers of length $\omega$. However, such formulas will not be considered in this paper.
\end{remark}

\begin{defini}{\label{effectivizer sum}}
    If $\phi$ and $\psi$ are $\in$-statements in prenex normal form, then we refer to an ordered pair $(F,G)$, consisting of an effectivizer $F$ of $\phi$ and an effectivizer $G$ of $\psi$, as an effectivizer of $\phi\oplus\psi$. 
\end{defini}

\begin{remark}
    This notation is a mere matter of convenience. Effectivizers for $\phi\oplus\psi$ are easily (and, in particular, effectively) converted into effectivizers for the prenex normal form of $\phi\wedge\psi$ and vice versa. Note, in particular, that this does \textit{not} define a \textit{formula} ``$\phi\oplus\psi$''.
\end{remark}

We recall how OTM-programs can be modified to allow for calls to (class) functions $F$ mapping sets of ordinals to sets of ordinals as explained in \cite{Ca2016}, Definition $4$: The Turing commands are supplemented by a new $F$ command, and the machine works with an additional tape $T$ (which, in \cite{Ca2016}, was dubbed ``miracle tape''). The machine can operate on $T$ as it does on any other tape; when the $F$-command is carried out, the set of ordinals $x$ coded by the current content of $T$ is replaced by $F(x)$.   To indicate that the $F$-OTM-program $P$ is run in the class function $G$ (mapping sets of ordinals to sets of ordinals) in the input $x$, we write $P^{G}(x)$. In this way, OTM-programs can work relative to effectivizers. 
In order to simplify our exposition, we refer to these ``$F$-OTM-programs'' as OTM-programs in this paper from now on.

The following now generalizes \cite{Ca2018}, Definition $4$ and Definition $5$ (which, in turn, is a variant of classical effective reducibility and Weihrauch reducibility for OTM-computability) to formulas of arbitrary quantifier complexity:

\begin{defini}{\label{reducibility}}
Let $\phi$, $\psi$ be $\in$-formulas in prenex normal form. We say that $\phi$ is \textit{OTM-reducible} to $\psi$, written $\phi\leq_{\text{OTM}}\psi$ if and only if there are an OTM-program $P$ and some $\rho\in\text{On}$ such that, for each encoded effectivizer $F_{\psi}$ of $\psi$, $P^{F_{\psi}}(\rho)$ computes an encoded effectivizer for $\phi$. If $P$ computes a $k$-effectivizer for $\phi$ from every $l$-effectivizer for $\psi$ for some $k,l\in\omega$, we write $\phi\leq_{\text{OTM}}^{k,l}\psi$.

We say that $\phi$ is \textit{ordinal Weihrauch reducible} to $\psi$, written $\phi\leq_{\text{oW}}\psi$, if and only if there are such $P$ and $\rho$ that $P^{\hat{F}}$ computes an effectivizer for $F_{\psi}$, but, for each input, only calls $\hat{F}$ (at most) once. Equivalently, there are OTM-programs $P$ and $Q$ such that, for any code $c_{x}$ for a set $x$ and any $\hat{F}$ as above, $Q(c_{x})$ outputs some $c_{z}$ such that $P$, when run in the input $\hat{F}(c_{z})$, outputs some $c_{y}$ coding a set (or a finite sequence of sets) $y$ with the required property. If this computes $k$-effectivizers for $\phi$ from $l$-effectivizers for $\psi$, we speak of ordinal $(k,l)$-Weihrauch reducibility, written $\phi\leq_{\text{oW}}^{k,l}\psi$.

These definitions also extends to effectivizers for schemes in the obvious way.
\end{defini}

\begin{remark}
In a ZFC framework, class functions must be definable; for the purpose of this paper, this is an unwanted restriction of possible effectivizers (this concerns in particular those results concerning the axiom of choice). In particular, reductions may become trivial when definable effectivizers do not exist or when they do exist, but, due to the lack of a global well-ordering, there are still no encoded effectivizers. The way to formalize reducibility in ZFC is then the following: We add a new unary function symbol $F$ to the language of ZFC and formulate the ZFC-axioms for this extended language (which modifies only separation and replacement); call the resulting theory $\text{ZFC}(F)$. We then add an axiom $\mathcal{A}$ (or, the case of a scheme, a scheme) expressing that $F$ is an effectivizer for the statement $\psi$. We say that $\phi$ is ZFC-provably OTM-reducible to $\psi$ if and only if, for some OTM-program $P$, $\text{ZFC}(F)+\mathcal{A}$ proves that, for each encoding $\hat{F}$ of $F$, $P^{\hat{F}}$ computes an effectivizer for $\phi$. Most of our results and arguments work equally well for both definitions; it is only in our treatment of AC below that we need to  switch to ZFC-provable OTM-reducibility.

\end{remark}

We note some immmediate properties of these notions.

\begin{prop}
    \begin{enumerate}
        \item For all $i,j,k\in\omega$, if $j\leq k$, then $\leq_{\text{OTM}}^{(i,j)}\subseteq \leq_{\text{OTM}}^{(i,k)}$.
        \item For all $i,j,k\in\omega$, if $j\leq k$, then $\leq_{\text{OTM}}^{(j,i)}\supseteq \leq_{\text{OTM}}^{(k,i)}$.
        \item $\leq_{\text{OTM}}^{(i,i)}$ is transitive, for all $i\in\omega$. 
        \item If $\phi$ and $\psi$ are $\Pi_{3}$, then $\phi\leq_{\text{OTM}}\psi$ is equivalent to $\phi\leq_{\text{OTM}}^{1,1}\psi$. 
    \end{enumerate}
\end{prop}
\begin{proof}
    (1)-(3) are trivial. For (4), note that a $\Pi_{3}$-formula possesses (at most) one existential quantifier that requires a witness.
\end{proof}

\section{Axioms of ZFC and other basic statements}

We will now consider the axiom (schemes) of ZFC, along with some basic consequences of ZFC, with respect to their OTM-effectivity and mutual OTM-reducibility.

\begin{defini}{\label{principles}}
We will use the following abbreviations for various set-theoretical statements:
\begin{itemize}
    \item $\bot$ denotes $0=1$, while $\top$ denotes $0=0$.
    \item HP: For each OTM-program $P$ and each ordinal parameter $\rho$, there is $i\in\{0,1\}$ such that either $P(\rho)\downarrow$ and $i=1$ or $P(\rho)\uparrow$ and $i=0$. 
    \item GreaterCard: For each $\alpha\in\text{On}$, there is a cardinal $\kappa>\alpha$. 
    \item NextCard: For each $\alpha\in\text{On}$, there is a next largest cardinal $\kappa>\alpha$; we will denote this cardinal as $\alpha^{+}$.
    \item Card: Every set has a cardinality. By a $1$-effectivizer for Card, we mean one that, for a given set $s$, returns its cardinality. On the other hand, a full effectivizer for Card will return a pair $(\kappa,f)$ consisting of the cardinality $\kappa$ and a bijection $f$ between $\kappa$ and $s$.
    \item OrdCard: Every ordinal has a cardinality. We use the analogous convention to Card.
    \item PowerCard: For each $\alpha\in\text{On}$, the cardinality of its power set exists (that is, there are a minimal ordinal $\kappa$ such that there is a bijection between $\kappa$ and $\mathfrak{P}(\alpha)$). We use the analogous convention to Card.
    \item DecCard: Every ordinal is or is not a cardinal.\footnote{This requires a careful formalization in order to make computational sense. The formalization we chose is the following: For every ordinal $\alpha$, there $i\in\{0,1\}$ such that either $i=0$ and $\alpha$ is not a cardinal, or $i=1$ and $\alpha$ is a cardinal.}
    \item Pot: $\forall{x}\exists{y}\forall{z}(z\in y\leftrightarrow z\subseteq x)$
    \item Sep$_{\phi}$: $\forall{X,p}\exists{Y}\forall{z}(z\in Y\leftrightarrow z\in X\wedge\phi(z,p))$, for all formulas $\phi$. By $\Sigma_{n}$-Sep, we denote the set of all formulas Sep$_{\phi}$ where $\phi$ is $\Sigma_{n}$, for $n\in\omega$. 
    The totality of all these formulas will be denoted by Sep.
    \item Rep$_{\phi}$: $\forall{X,p}(\forall{x\in X}\exists!{y}\phi(x,y,p)\rightarrow\exists{Y}\forall{x\in X}\exists{y\in Y}\phi(x,y,p))$; by $\Sigma_{n}$-Rep, we denote the set of all formulas Sep$_{\phi}$, where $\phi$ is $\Sigma_{n}$, for $n\in\omega$.
    The totality of all these formulas will be denoted by Rep.
    \item Coll$_{\phi}$: $\forall{X,p}(\forall{x\in X}\exists{y}\phi(x,y,p)\rightarrow\exists{Y}\forall{x\in X}\exists{y\in Y}\phi(x,y,p))$; by $\Sigma_{n}$-Coll, we denote the set of all formulas Coll$_{\phi}$, where $\phi$ is $\Sigma_{n}$, for $n\in\omega$; the totality of all these formulas will be denoted by Coll.
    \item AC: $\forall{X}(\emptyset\notin X\rightarrow\exists{f:X\mapsto\bigcup{X}}\forall{x\in X}f(x)\in x)$.
    \item For $n\in\omega$, let $T_{n}(k,p)$ be a truth-predicate for $\Sigma_{n}$-formulas in the language of set theory. That is, for a natural enumeration $(\phi_{i}:i\in\omega)$ of the $\Sigma_{n}$-formulas, a tuple $p$ of sets, $T_{n}(k,p)$ is true if and only if $\phi_{k}(p)$.\footnote{To simplify the notation, a finite sequence of parameters is encapsulated in a single parameter $\rho$, if $\phi$ has several free variables.} This can be expressed by an $\in$-formula $\exists{x_{0}}\forall{y_{0}}...\exists{x_{n}}\forall{y_{n}}\psi_{n}^{T}(x_{0},y_{0},...,x_{n},y_{n})$ with $\psi^{T}$ quantifier-free; likewise, $\neg T_{n}$ can be expressed as $\forall{v_{0}}\exists{w_{0}}...\exists{v_{n}}\forall{w_{n}}\neg\psi_{n}^{T}(x_{v},w_{0},...,v_{n},w_{n})$, where $\{v_{0},w_{0},...,v_{n},w_{n}\}$ is disjoint from $\{x_{0},y_{0},...,x_{n},y_{n}\}$. Let $Q$, $Q_{\neg}$ be the quantifier sequence starting these formulas. By $\Sigma_{n}$-TND,\footnote{``tertium non datur''. } we denote the of the statement $$\forall{k}\forall{p}\exists{i\in\{0,1\}}QQ_{\neg}((i=1 \wedge \psi_{n}^{T}(k,p))\vee (i=0\wedge\neg \psi_{n}^{T}(k,p))).$$
    
\end{itemize}
\end{defini}

\begin{remark}
This list contains several axiom (schemes) of ZFC. It is easy to see that all ZFC-axioms not mentioned in this list are OTM-effective and can thus be disregarded in considerations about reducibility.
\end{remark}

It is not clear that computations relative to effectivizers of full Rep and Sep exist in general, since the straightforward way of defining them would involve a truth predicate for the universe. We can, however, work around this by working in Feferman set theory,\footnote{See \cite{Feferman}, p. 202.} which is known to be equiconsistent with ZFC (ibid.) and extends ZFC by the statement that, for a proper class $C$ of cardinals $\kappa$, $V_{\kappa}$ is an elementary submodel of $V$. Then, the partial computation of length $\alpha$ of an OTM-program $P$ relative to, say, a separation operator, can be seen to exist by picking $\kappa\in C\setminus(\alpha+1)$, defining the computation with quantifiers restricted to $V_{\kappa}$, and using the fact that $V_{\kappa}\prec V$. We thus assume, in all results concerning Sep and Rep in this paper, Feferman set theory, which will be indicated by writing (Fef) at the beginning of the respective statement.

\begin{remark}
Clearly, HP is not OTM-effective, as an OTM-effectivizer would amount to a halting problem solver. It immediately follows that not even all logical tautologies are OTM-effective, so that every axiomatic system will, in classical logic, prove some non-OTM-effective statements, no matter the specific axioms. In particular, classical logic does not preserve effectiveness.
\end{remark}

\subsection{Realizability, reducibility, and truth}

It is easy to see that every OTM-effective statement must be true, for the OTM-computability of witnesses implies a fortiori the existence of witnesses. 
We already saw that HP, although true on the basis of classical first-order logic, is not OTM-effective. Thus, we note: 

\begin{proposition}{\label{truth and effectivity}}
OTM-effectivity is strictly weaker than classical truth. 
\end{proposition}

We note that, under $V=L$, truth is, however, equivalent to the existence of a \textit{definable} (but not necessarily OTM-computable) effectivizer:

\begin{proposition}{\label{definable effectivizers and L}}
In ZFC+$V=L$, it is provable, for every $\in$-formula $\phi$, that $\phi$ is equivalent to the existence of a definable effectivizer for $\phi$.
\end{proposition}
\begin{proof}
    This follows from the existence of definable Skolem functions in $L$, see, e.g., Schindler and Zeman \cite{SZ}.
\end{proof}

\begin{remark}
    This is not true in general, however; we are thus led to the general notion of ``definable effectivizer'' and, correspondingly, ``definable reducibility'', which results from the definitions of OTM-effectivizers and OTM-reductions by replacing OTM-computability with set-theoretical definability, and which we plan to take up in future work.
\end{remark}



We note that provability of the implication and OTM-realizability are both different from OTM-effective reducibility. The definition of OTM-realizability can be found in \cite{CGP}, Definition $17$. (Note that, in this definition, the OTM-realizer for a universally quantified statement $\forall{x}\psi(x)$ has access to $a$ when computing a realizer for $\psi(a)$.)

\begin{lemma}{\label{reducibility and provable implication}}
\begin{enumerate}
\item There are $\in$-sentences $\phi$ and $\psi$ such that $\phi\leq_{\text{OTM}}\psi$, but $\psi\rightarrow\phi$ is not provable in ZFC.
\item There are $\in$-sentences $\phi$ and $\psi$ such that $\psi\rightarrow\phi$ is provable in ZFC, but $\phi\nleq_{\text{OTM}}\psi$.
\item The same, in fact, holds for every consistent $\in$-theory that is sufficiently strong to formalize the relevant statements; in particular, this includes extensions of KP.
\item There are statements $\phi$ and $\psi$ in prenex normal form such that $\phi\rightarrow\psi$ is OTM-realizable, but $\psi\nleq_{\text{OTM}}\phi$.
\item If $\phi$ and $\psi$ are statements in prenex normal form such that $\psi\leq_{\text{OTM}}\phi$, then $\phi\rightarrow\psi$ is (uniformly) OTM-realizable.
\end{enumerate}
\end{lemma}
\begin{proof}
\begin{enumerate}
\item Let $\phi$ be HP and let $\psi$ be $0=1$. Then $\psi\leq_{\text{OTM}}\phi$ is vacuously true (since HP has no effectivizers), but, since $\phi$ is a logical tautology, $\phi\rightarrow\psi$ is not provable in ZFC (unless ZFC is inconsistent).
\item An obvious example is $0=0\rightarrow\text{HP}$.
\item The same arguments work in the more general case.
\item Let $\phi$ be NextCard (see Definition \ref{principles}) and $\psi$ be $0=1$. By \cite{CGP}, Lemma $41$, 
 NextCard is not OTM-realizable (since OTM-programs cannot output objects whose cardinality is larger than the length of their input and the size of their parameter), so $\phi\rightarrow\psi$ is trivially OTM-realizable. But NextCard has the (obvious) effectivizer $\kappa\mapsto\kappa^{+}$, so we cannot have $\psi\leq_{\text{OTM}}\phi$. 
\item Let $r$ be an OTM-realizer for $\phi$. Then $r$ induces, in an OTM-computable way, an OTM-computable effectivizer $F$ for $\phi$. By definition of $\psi\leq_{\text{OTM}}\phi$, we can compute an effectivizer $G$ for $\psi$ relative to $F$. By combining the reduction with the realizer for $\phi$, we obtain a realizer for $G$. 
\end{enumerate}    
\end{proof}

\begin{remark}
    We note that (5) depends crucially on the formula's being given in prenex normal form. 
\end{remark}



\section{Reducibility between basic set theoretical principles}

The main goal of this paper is to determine which reducibilities hold between the statements introduced above. For this, it will come in handy that evaluating truth predicates reduces to separation operators:

\begin{lemma}{\label{truth and sep}}
\begin{enumerate}
\item For all $n\in\omega$, there is an OTM-program which, relative to any effectivizer $F_{\text{n-Sep}}$ for $\Sigma_{n}$-separation, decides truth of $\Sigma_{n}$-formulas. 

\item (Fef) There is an OTM-program which, relative to any effectivizer $F_{\text{Sep}}$ for Sep, decides truth of arbitary $\in$-formulas. 
\end{enumerate}
In fact, the reduction can be done by calling $F_{\text{Sep}}$ only once.
\end{lemma}
\begin{proof}
Given a $\Sigma_{n}$-formula $\phi$ and a parameter $\vec{p}$, obtain $F_{\text{n-Sep}}(\lceil\phi(\vec{p})\rceil,\{\emptyset\})$, then return ``true'' if and only if the result is non-empty, and ``false'' otherwise.
\end{proof}


For the record, we note that this also works in the reverse direction by the obvious strategy: 

\begin{lemma}{\label{sep and truth}}
Let $F_{\Sigma_{n}-\text{truth}}$ be a (class) function which, for every $\Sigma_{n}$-formula $\phi(\vec{x})$ with free variables in $\vec{x}$ and any tuple $\vec{a}$ of sets of the right length, $$F_{\Sigma_{n}-\text{truth}}(\phi,\vec{a})=\begin{cases}1\text{, if }\phi(\vec{a})\\0\text{, otherwise}\end{cases}.$$ Then an effectivizer for $\Sigma_{n}$-separation is OTM-computable relative to $F_{\Sigma_{n}-\text{truth}}$.
\end{lemma}
\begin{proof}
    Given a code for a set $b$, a $\Sigma_{n}$-formula $\phi(x,\vec{p})$ and a tuple $\vec{a}$ of sets, run through $b$, applying $F_{\Sigma_{n}-\text{truth}}$ to $(\phi,(c,\vec{a}))$ for every $c\in b$, and collecting together those elements for which the answer is positive.
\end{proof}

\begin{remark}
    Note that, in contrast to the reduction of truth to separation, the reduction described here does not work with a single call to the effectivizer; in fact, on input $x$, it will make $\text{card}(x)$ many such calls. As we will see in a separate paper (\cite{reduction complexity}), this is optimal.
\end{remark}

\begin{lemma}{\label{order properties}}
$\leq_{\text{OTM}}$ is a reflexive, transitive order on the set of $\in$-sentences in prenex normal form with minimal element $\top$ and maximal element $\bot$.
Moreover, for any $\phi<_{\text{OTM}}\bot$, there is $n\in\omega$ such that $\phi<_{\text{OTM}}\Sigma_{n}-\text{TND}$.
\end{lemma}
\begin{proof}
All claims except for the last one are trivial. 

To see the last claim, let $\phi:\Leftrightarrow\forall{x_{0}}\exists{y_{0}}...\forall{x_{n}}\exists{y_{n}}\psi(x_{0},y_{0},...,x_{n},y_{n})$, where $\psi$ is quantifier-free (in order to simplify our notation, we suppress parameters). Assume that $\bot\nleq_{\text{OTM}}\phi$, and let $F$ be an effectivizer for $\Sigma_{n}$-TND. This implies that $\phi$ has an effectivizer (though not necessarily an OTM-computable one), as every statement trivially reduces to a statement that has no effectivizers. In particular then, $\phi$ is true. Let $k$ be the index of $\phi$ in the enumeration of formulas. Thus, $F((k))=1$. Letting $G((a_{0},...,a_{j})):=F((k,a_{0},...,a_{j}))$ for $j\leq n$ will then yield the desired effectivizer for $\phi$.

\end{proof}


We now start by noting some reducibility results concerning the cardinality principles.

\begin{proposition}{\label{card and hp}}
\begin{enumerate}
    \item GreaterCard$\leq_{\text{oW}}$NextCard.
    \item NextCard$\equiv_{\text{OTM}}$DecCard
    \item NextCard$\equiv_{\text{OTM}}$GreaterCard is independent of ZFC.
    \item Card$\equiv_{\text{oW}}$OrdCard
    \item Card$\equiv_{\text{OTM}}$DecCard
    \item HP$\leq_{\text{OTM}}$GreaterCard.
    \item Card$\leq_{\text{oW}}^{1,1}$Pot
    \item Card$\leq_{\text{OTM}}$Pot is independent of ZFC.
    \item NextCard$\leq_{\text{OTM}}\Sigma_{1}$-Sep.
    \item GreaterCard$\nleq_{\text{oW}}^{3,1}\Sigma_{n}$-Sep for all $n\in\omega$.
    \item If $V=L$, then HP$\equiv_{\text{OTM}}$GreaterCard
\end{enumerate}
(In particular, NextCard and GreaterCard are not OTM-effective.)\footnote{This statement, which amounts to the undecidability of the class of cardinals on an OTM, is well-known in the literature; the first explicit mention we know of is in Dawson \cite{Dawson}.}
\end{proposition}
\begin{proof}
\begin{enumerate}
    \item Since the cardinal successor of an ordinal $\alpha$ is in particular a greater cardinal than $\alpha$, any effectivizer for NextCard is also one for GreaterCard.
    \item Let $F$ be an effectivizer for DecCard, and let $\alpha$ be an ordinal. Starting with $\alpha+1$, run through the ordinals and apply $F$ to each of them until the output is $1$; then return the first ordinal for which this is the case.

    On the other hand, let an effectivizer $F$ for NextCard and an ordinal $\alpha$ be given. Relative to $F$, the sequence of cardinals is computable by starting with $0$, applying $F$ at successor levels and taking unions at limit levels. Compute this sequence up to the point where it contains an element that is strictly larger than $\alpha$ (which will always happen at a successor level); then return its predecessor. 
    \item By (1), it only remains to consider NextCard$\leq_{\text{OTM}}$GreaterCard.

    This is true in $L$: Let $F_{\text{gc}}$ be an effectivizer for GreaterCard, and let $\alpha$ be an ordinal. Then $F((\alpha))$ will be a cardinal greater than $\alpha$. Using standard techniques for OTMs (cf., e.g., \cite{CarlBuch}, Lemma 3.5.3), 
    we can then compute a code for $L_{F((\alpha))}$ from $F(\alpha)$. By basic fine-structure theory, an ordinal $\beta$ will be a cardinal in $L_{F((\alpha))}$ if and only if it is a cardinal in $L$. We can thus search through $L_{F((\alpha))}$, using the ability of OTMs to evaluate set-theoretic truth in a given structure (cf., e.g., Koepke, \cite{Koepke}, Lemma 6.1 or \cite{CarlBuch}, Theorem 2.3.28) to check whether $L_{F((\alpha))}$ contains a cardinal $\kappa$ above $\alpha$. If so, we return $\kappa$ as the next cardinal after $\alpha$; otherwise, we return $F((\alpha))$.

    We now show how to obtain a model of ZFC in which NextCard$\nleq_{\text{OTM}}$GreaterCard. 
    First, use Cohen-forcing to obtain a real number $g\subseteq\omega$ that is Cohen-generic over $L$. The idea is now to code $g$ into the cardinal structure of $M$. To this end, consider a forcing extension $M$ of $L$ in which, for all $k\in\omega$, $(\aleph_{2k+1})^{M}=((\aleph_{2k}^{M})^{+})^{L}$ if and only if $k\in g$, and otherwise, $(\aleph_{2k+1})^{M}=((\aleph_{2k}^{M})^{++})^{L}$. That is, the cardinal successor of $\aleph_{2k}^{M}$ in $M$ is the next $L$-cardinal after $\aleph_{2k}^{M}$ if $k\in g$, and otherwise, it is the second-next, while all other cardinals remain as in $L$. Such an extension of $L$ can, e.g., be obtained by Easton-forcing (see, e.g., \cite{Kunen}, p. 262f). 
    Let us assume for a contradiction that, in $M$, $(P,\rho)$ effectively reduces NextCard to GreaterCard. Now, let $F$ be the class function which sends an ordinal $\alpha\in[\aleph_{k}^{M},\aleph_{k+1}^{M})$ to $(\aleph_{2k})^{L}$ for $k\in\omega$ and all other ordinals $\alpha$ to $(\alpha^{+})^{L}$. Then, in $M$, $F$ is an effectivizer of GreaterCard. By definition, $F$ is definable in $L$. Thus, if $G$ is the effectivizer of NextCard obtained via $(P,\rho)$ from $F$, then $G$ is definable in $L$ as well. Then, using the definition of NextCard, we can define in $L$ the sequence $s:=(\aleph_{k}^{M}:k\in\omega)$. But then, relative to $s$, $g$ is definable (by definition of $s$), so we have $g\in L$, a contradiction. Hence $M\models$NextCard$\nleq_{\text{OTM}}$GreaterCard, as desired.
    \item Clearly, OrdCard$\leq_{\text{oW}}$Card, as every instance of OrdCard is an instance of Card. For the reverse direction, let a code $c_{a}$ for a set $a$ be given. This in particular encodes an injection $f$ from $a$ into the ordinals. From $f$, we can easily compute the order type $\gamma$ of $f[a]$. Finally, we have $\text{card}(a)=\text{card}(\gamma)$.
    \item To reduce DecCard to Card, just apply a Card-effectivizer to a given ordinal $\alpha$ and then return $1$ if the output agrees with $\alpha$ and $0$, otherwise. This actually shows DecCard$\leq_{\text{oW}}$Card.
    
    To reduce Card to DecCard, let a set $a$ be given. By (3), we can assume without loss of generality that $a$ is an ordinal. Let us reserve an extra scratch tape $t$ initially filled with zeroes. Now, starting with $0$, successively run through the ordinals, using the DecCard-effectivizer in each step to check whether that ordinal $\alpha$ is a cardinal, and, if it is, write $1$ to the first $\alpha$ many cells of $t$. Once the computation arrives at $a+1$, exactly the first $\text{card}(a)$ cells of $t$ will be filled with $1$s, which represents a code for $\text{card}(\alpha)$. 
    \item It is well-known (see, e.g., the proof of \cite{CarlBuch}, Lemma 8.6.3) that, for an OTM-program $P$, a parameter $\rho\in\text{On}$ and an infinite cardinal $\kappa>\rho$, $P(\rho)$ will either halt in less than $\kappa$ many steps or it will not halt at all. Thus, given an effectivizer for GreaterCard, it suffices to run $P(\rho)$ for $\text{max}\{F((\rho)),F((\omega))\}$ many steps to determine whether it will halt.
    \item Let $\hat{F}$ be an encoded effectivizer for Pot, and let $x$ be an arbitrary set. To compute its cardinality, we use $\hat{F}$ to obtain a code for $p:=\mathfrak{P}(x\times x)$. This set will in particular contain all well-orderings of $x$. It is well-known\footnote{This was originally shown in Hamkins and Lewis, \cite{Hamkins-Lewis}, Theorem 2.1, and then easily generalizes to higher order types (see, e.g., \cite{CarlBuch}, Theorem 2.3.25).} that OTMs can check sets for being well-orderings. Thus, we can run through $p$, identify all elements which are well-orderings of $x$, and compute their order-type. Finally, we return the minimum of the order-types thus obtained.
    \item If $V=L$, we can, as in the proof of (4), compute the set of all well-orderings of $x$ of minimal length and then output the $<_{L}$-minimal of these. Thus, it is not provable in ZFC that Card is not OTM-reducible to Pot.

    In general, a full effectivizer for Card would allow one to compute a well-ordering for an arbitrary given set; thus, it would allow for an effectivizer of the well-ordering principle WO. Since AC$\leq_{\text{oW}}$WO (see \cite{Ca2018}, Proposition 13.3), and this reduction is provable in ZFC, the provability of the reducibility of Card to Pot in ZFC would imply that AC provably in ZFC reduces to Pot. However, as we will see below (Proposition \ref{choice and friends}(3)), this is not the case.


    \item Let $\alpha\in\text{On}$ be given. The statement ``$x$ not a cardinal'' is $\Sigma_{1}$ (it states the existence of a bijection between $x$ and an element of $x$). By the reducibility of $\Sigma_{1}$-truth to $\Sigma_{1}$-Sep, we can, relative to an effectivizer $F$ for $\Sigma_{1}$-Sep and starting with $\alpha+1$, run through the ordinals, test each one with $F$ for not being a cardinal, and output the first ordinal for which the answer is negative. 
    \item By Lemma \ref{oW and cardinality},\footnote{We sincerely apologize for referring to a lemma succeeding the statement in question; however, the statement best belongs in this context, while moving the Lemma \ref{oW and cardinality} up would disrupt the flow of the paper.} a construction that raises cardinals -- such as GreaterCard -- cannot be oW-reducible to a construction that does not. Since the set $S$ obtained from a given set $X$ via separation will satisfy $\text{tc}(S)\subseteq\text{tc}(X)$, this suffices to prove the statement.
    \item If $V=L$, consider the OTM-program $P_{\text{card-check}}$ that, in the input $\alpha$, enumerates $L$,\footnote{Such a program exists by the work of Koepke; see, e.g., \cite{Koepke}.} searches for a bijection that collapses $\alpha$ to some smaller ordinal and halts once such an ordinal is found. This program will terminate if and only if $\alpha$ is a cardinal in $L$. Thus, giving access to an effectivizer of HP, an OTM can identify $L$-cardinals. Now, given some ordinal $\alpha$, we can simply count upwards from $\alpha$, checking each arising ordinal for being an $L$-cardinal, and halting once we have found one.
\end{enumerate}
\end{proof}

\begin{remark}
    It is worthwhile to briefly ponder the contrast between (4) and (5) of Theorem \ref{card and hp}: While one application of Pot is enough to determine the cardinality of an arbitrary set, this does not mean that we can effectively compute well-orderings. In other words, we can compute the right cardinal, but we cannot pick a bijection witnessing that it is indeed the right one. The difficulty here is that, although $\mathfrak{P}(x\times x)$ will contain many well-orderings of $x$ of minimal length, we have in general no way to chose one of these. We will come back to this in our treatment of the axiom of choice below.
\end{remark}

\begin{remark}
    To reduce OrdCard (or Card) to DecCard, the reduction described in the above proof applies DecCard $\text{card}(a)+1$ many times in the input $a$. Can we get away with less applications? How many are needed? Such questions concerning \textit{reduction complexities} are treated in \cite{reduction complexity}.

\end{remark}

In \cite{Ca2018}, Lemma $7$, it was stated with only a rather sketchy argument that relations that raise cardinalities cannot be OTM-reducible to relations that do not.\footnote{After writing this article, it was brought to our attention that this fact was already pointed out in the Master's thesis of Jasper Stammes, see Lemma 4.1 on p. 33 of \cite{StammesMaster}.} Unfortunately, as Lemma \ref{powercard}(5) below demonstrates, this is false; the statement is true for oW-reducibility, but not for OTM-reducibility. (We note that this still covers all applications of Lemma $7$ made in \cite{Ca2018}.) We give here the corrected version (which is a variant of Hodges' ``cardinality method'' explained in \cite{Hodges}), with a more careful argument.



\begin{defini} (Cf. Hodges, \cite{Hodges}, p. 145)
    We say that a partial (class) function $F:V\rightarrow \text{On}$ \textit{raises cardinalities} 
     if and only if, for any cardinal $\kappa$, there is a set $x$ such that $\text{card}(\text{tc}(x))>\kappa$ and $\text{card}(F(x))>\text{card}(\text{tc}(\{x\}))$. 
\end{defini}

\begin{lemma}{\label{oW and cardinality}} 
Let $F$ and $G$ be partial class functions from $V$ to $\text{On}$. If $G$ raises cardinalities and $G$ is OTM-computable relative to $F$ by a parameter-program $(P,\rho)$ such that $P^{F}(a,\rho)$ makes only finitely many calls to $F$, then $F$ raises cardinalities.
\end{lemma}
\begin{proof}
    We show that a computation with a single application of an extra (class) function that does not raise cardinalities; the extension to finitely many applications then follows inductively. Concerning one application, let $F$ and $G$ be as in the assumptions of the lemma, and let $(P,\rho)$ be a parameter-program. By assumption, pick a cardinal $\kappa>\rho$ such that, for some set $a$, we have $\lambda:=\text{card}(\text{tc}(a))>\kappa$ and $\text{card}(G(a))>\text{card}(\text{tc}(a))$. Consider the computation of $P^{F}(c(a),\rho)$, where $c(a)$ is a code for $a$. It is well-known (cf., e.g., the proof of Lemma 8.6.3 in \cite{CarlBuch}) that, without the extra function, this computation will, if it halts, have a length of cardinality $\leq\lambda$ ($\ast$). Since we could modify $P$ slightly to halt when $P$ makes a call to $F$, we can conclude that the length of the computation up to the first call to $F$ has cardinality $\leq\lambda$. Thus, at most $\lambda$ many cells have been written on at this time, and so the set $b$ to which $F$ is applied has a transitive closure of cardinality $\leq\lambda$. If $F$ would not raise cardinals, the same would be true of $F(b)$. Another application of ($\ast$) yields that the OTM-computation in the parameter $\rho$ and the input $F(b)$ (which, by assumption, makes no further calls to $F$) will also make $\leq\lambda$ many steps and thus have an output whose transitive closure has cardinality at most $\lambda$, contradicting the assumption that $\text{card}(G(a))>\text{card}(\text{tc}(a))$. 
\end{proof}

\begin{lemma}{\label{powercard}}
\begin{enumerate}
    \item PowerCard$\leq_{\text{OTM}}^{1,1}$NextCard is independent of ZFC.
    \item GreaterCard$\leq_{\text{oW}}^{1,1}$PowerCard
    \item PowerCard$\leq_{OTM}^{1,1}$Pot
    \item PowerCard$\leq_{\text{OTM}}^{1,1}\Sigma_{2}$-Sep.
    \item PowerCard$\nleq_{\text{oW}}^{1,1}\Sigma_n-$Sep, for all $n\in\omega$.
\end{enumerate}
\end{lemma}
\begin{proof}
\begin{enumerate}
    \item    Clearly, in $L$ (or any other model of GCH), we have PowerCard$\equiv_{\text{oW}}$NextCard, as the continuum function is identical with the cardinal successor function.
    
    The other part is proved similarly to Proposition \ref{card and hp}(2), this time by coding too much information into the continuum function: Let $g\subseteq\omega$ again be Cohen-generic over $L$, and let $M$ be a generic extension of $L$ that has the same cardinals as $L$, but satisfies $$2^{\aleph_{k}}=\begin{cases}\aleph_{2k}\text{, if }k\in g\\\aleph_{2k+1}\text{, otherwise}\end{cases}$$ for all $k\in\omega$. Such an extension exists by \cite{Kunen}, p. 212-213. Suppose that $(P,\rho)$ reduces PowerCard to NextCard in $M$. Let $F$ be the class function mapping each ordinal $\alpha$ to $(\alpha^{+})^{L}$. Then $F$ is an effectivizer of NextCard in $M$. Moreover, $F$ ist definable in $L$. Hence, the class function $G$ computed by $P^{F}(\rho)$ is also definable in $L$. But from $G$, we can define $g$, so $g\in L$, again a contradiction. 
    \item Trivial, as the cardinality of the power set of an ordinal $\alpha$ is in particular a cardinal larger than $\alpha$. 

    \item Let an ordinal $\alpha$ be given, and let $F$ be an effectivizer for Pot. The computation starts by obtaining $F(\alpha)$. Now, by Proposition \ref{card and hp}(4), the cardinality of $F(\alpha)$ can be obtained relative to $F$. 
    \item Let an ordinal $\alpha$ be given. The computation runs through the ordinals, and, for every ordinal $\beta$, applies the effectivizer for $\Sigma_{2}$-separation to the set $\{\beta\}$ and the formula $\phi(x):\Leftrightarrow\exists{f}\exists{y}\forall{z}((z\in y\leftrightarrow z\subseteq x)\wedge f:\beta\overset{1:1}{\rightarrow}x)$ (which expresses ``$\beta$ is the cardinality of the power set of $x$''). If the result is empty, we continue with the next $\beta$. Otherwise, we halt and output $\beta$. 
    \item Immediate from Lemma \ref{oW and cardinality} and the fact that separation does not the increase the cardinality of the transitive closure of a set.
\end{enumerate}
\end{proof}

\subsection{NextCard and PowerCard}

By a somewhat more involved version of this argument, we can also show that NextCard does not ZFC-provably reduce to PowerCard.

We recall a particular (standard) forcing notions for collapsing a cardinal.

\begin{defini}
For $\alpha\in\text{On}$, $\alpha>0$, we denote by $\mathbb{P}_{\alpha}$ the Levy collapse for $\aleph_{\omega\alpha+1}$ and $\aleph_{\omega\alpha+2}$, consisting of partial functions from $\aleph_{\omega\alpha+1}$ to $\aleph_{\omega\alpha+2}$ of cardinality $<\aleph_{\omega\alpha+1}$.
\end{defini}

The necessary (standard) prerequisties from forcing for our construction are summarized in the following lemma.\footnote{We thank  Philipp Schlicht for helpful discussions on the (iterated) forcing constructions used below.} 


\begin{lemma}{\label{one iteration step}}
Let $\alpha>0$ be an ordinal, and let $\kappa=\aleph_{\omega\alpha}$. Let $M$ be a model of ZFC such that $2^{\kappa^{+}}=2^{\kappa^{++}}=\kappa^{+++}$ for all limit cardinals $\kappa$, while $2^{\kappa}=\kappa^{+}$ for all cardinals that are not successors of limit cardinals. 
Let $G$ be generic for $\mathbb{P}_{\alpha}$ over $M$. Then $M[G]$ has the following properties:
\begin{enumerate}
\item $\text{card}^{M[G]}(\kappa^{+})=\text{card}^{M[G]}(\kappa^{++})$ (i.e., $\kappa^{++}$ is collapsed). 
\item For every ordinal $\alpha\notin[\kappa^{++},\kappa^{+++})$, $\text{card}^{M}(\alpha)=\text{card}^{M[G]}(\alpha)$ (i.e., all other cardinals in $M$ are preserved). 
\item For every ordinal $\alpha$, we have $[\text{card}(\mathfrak{P}(\alpha))]^{M}=[\text{card}(\mathfrak{P}(\alpha))]^{M[G]}$ (i.e., all cardinalities of power sets are preserved).
\end{enumerate}
\end{lemma}
\begin{proof}

The forcing we consider is just the standard forcing for collapsing a cardinal, see, e.g., \cite{Jech}, p. 237f. Since $\kappa^{++}$ is regular, we have $(\kappa^{++})^{<\kappa^{+}}=(\kappa^{++})^{\kappa}=\kappa^{++}$, so this forcing collapses $\kappa^{++}$ to $\kappa^{+}$ while preserving all other cardinals by \cite{Jech}, Lemma 15.21.

As $\mathbb{P}_{\alpha}$ is $\kappa^{+}$-closed, it follows from \cite{Kunen}, Theorem 6.14, that $\mathfrak{P}^{M}(\beta)=\mathfrak{P}^{M[G]}(\beta)$ for every $\beta<\kappa^{+}$. Thus, (3) holds for all ordinal $<\kappa^{+}$.\footnote{Note that there is no need to indicate in which model the successor is taken, as the successor of $\kappa$ in $M$ and in $M[G]$ is the same.} 

So let $\beta\geq\kappa^{+}$, and let $\dot{\beta}$ be the canonical name for $\beta$. We will now count how many nice names for subsets of $\dot{\beta}$ (see \cite{Kunen}, Definition 5.11 and Lemma 5.12) there can be in $M$. First, $\mathbb{P}_{\alpha}$ has cardinality $(\kappa^{++})^{\kappa}=\kappa^{++}$ in $M$, and thus satisfies the $\kappa^{+++}$-cc. 
Consequently (using the assumptions about $M$), there are at most $(\kappa^{++})^{\kappa^{++}} = 2^{\kappa^{++}} = \kappa^{+++}$ many antichains of $\mathbb{P}_{\alpha}$ in $M$. 
It follows that, in $M$, the cardinality of the set of nice names for subsets of $\beta$ is bounded by $(\kappa^{+++})^{\text{card}^{M}(\beta)}$. 
If $\text{card}^{M}(\beta)\leq(\kappa^{++})^{M}$, this will be $\kappa^{+++}$, which, in $M$, agrees with $2^{\kappa^{+}}$ and $2^{\kappa^{++}}$ by assumption. 
If $\beta>\kappa^{++}$, then $(\kappa^{+++})^{\text{card}^{M}(\beta)}=2^{\text{card}^{M}(\beta)}$, which again agrees with $\mathfrak{P}^{M}(\beta)$. 
Thus, we have (3).

\end{proof}


\begin{lemma}{\label{nextcard powercard}}
     NextCard$\leq_{\text{OTM}}$PowerCard is independent of ZFC 
\end{lemma}
\begin{proof}
In $L$ (and any other model of GCH), NextCard and PowerCard coincide, so either reduces to the other.  

To see the consistency of NextCard$\nleq_{\text{OTM}}$PowerCard, use an iterated forcing to sabotage all candidates $(P,\rho)$ for a reduction: 
To prevent $(P,\rho)$ from witnessing the reduction, we change some cardinal while preserving all values of the continuum function. 

Let $M$ be a model of ZFC satisfying the condition for $M$ in Lemma \ref{one iteration step}; that is, $M$ satisfies GCH everywhere except at successor cardinals $\lambda^{+}$ of limit cardinals, for which we have $2^{\lambda^{+}}=\lambda^{+++}$.  Let $F$ be an effectivizer for PowerCard in $M$. 



The idea is the following: Consider parameter-program $(P_{k},\rho)$. Let $\kappa=\aleph_{\omega(\omega\rho+k)}$. Then, if $P_{k}^{F}((\kappa^{+})^{M},\rho)$ does not halt with output $(\kappa^{++})^{M}$ (which includes the case that it does not halt at all), we do not change $M$. But if it does, we use $\mathbb{P}_{\alpha}$ to collapse $\kappa^{++}$ to $\kappa^{+}$ without changing any values of the continuum function. In the generic extension $M[G]$, $F$ will thus still be an effectivizer for PowerCard, and the computation $P_{k}^{F}((\kappa^{+})^{M},\rho)$ will work exactly the same in $M[G]$ and thus still halt with output $(\kappa^{++})^{M}$, which, however, is not the cardinal successor of $\kappa^{+}$ in $M[G]$. 

Formally, let $A:=\{\omega\rho+k:P_{k}^{F}(\aleph_{\omega(\omega\rho+k)+1}^{M},\rho)\downarrow = \aleph_{\omega(\omega\rho+k)}^{M}\}$, and let $\mathbb{Q}$ be the Easton product forcing consisting of all forcings $\mathbb{P}_{\alpha}$ with $\alpha\in A$. This product is progressively closed by \cite{Reitz}, Lemma 117, and thus tame by \cite{Reitz}, Theorem 98.\footnote{Reitz shows this for BGC, every model of ZFC can be extended to one of BGC (\cite{Reitz}, pp. 60f), and BGC is conservative over ZFC (ebd.), so this guarantees that ZFC holds in the generic extension if it holds in the ground model.} 

Thus, if $G$ is $\mathbb{Q}$-generic, then $M[G]\models\text{ZFC}$, $F$ is still an effectivizer for PowerCard in $M[G]$, so all OTM-computations using $F$ are still computations relative to PowerCard in $M[G]$, but, by definition of $M[G]$, there is no pair $(P_{k},\rho)$ such that $P_{k}^{F}$ computes the cardinal successor of $\aleph_{\omega(\omega\rho+k)+1}^{M[G]}$ in $M[G]$ correctly. Thus, NextCard is not reducible to PowerCard in $M[G]$. 

\end{proof}

\begin{remark}
At this point, one might ask whether the assumption that $\text{NextCard}\leq_{\text{OTM}}^{(1,1)}\text{PowerCard}$ has any implications for the continuum function. The proof of Proposition \ref{card and hp}(3), which shows that we have $\text{NextCard}\leq_{\text{OTM}}^{1,1}\text{GreaterCard}$ and thus $\text{NextCard}\leq_{\text{OTM}}^{1,1}\text{PowerCard}$ in $L$ actually only requires 
that Card$=$Card$^{L}$, i.e., that the universe has the same cardinals as $L$. By using Easton-Forcing over $L$, we can thus generate models of ZFC in which the continuum function behaves in any possible way on the regular cardinals, while preserving the reduction of NextCard to PowerCard. 

In particular, the implication from GCH to the reduction cannot be reversed; in fact, assuming that the reduction is possible tells us nothing about the powers of regular cardinals. 
\end{remark}




\subsection{OTM-Reducibility and HOD}

We now consider reductions to separation operators. For this purpose, it is helpful to establish connections between OTM-computability relative to separation operators and HOD, which are also of independent interest. For a set $x$, we denote by $\text{HOD}(x)$ the class of sets that are heriditarily ordinal definable from $x$, see, e.g., \cite{Jech}, pp. 194f.

\begin{lemma}{\label{sep and HOD}}
Let $n\in\omega$, and let $x$, $y$ be sets. Suppose that $P$ is an OTM-program such that, for every effectivizer $F$ of $\Sigma_{n}$-Sep and every code $c_{x}$ for $x$, $P^{F}(c_{x})$ computes a code for $y$. Then $y\in\text{HOD}(x)$.
\end{lemma}
\begin{proof}
For the sake of simplicity, we only consider the unrelativized case where $x=\emptyset$. We first show that $\Sigma_{n}$-Sep has a definable encoded effectivizer, for every $n\in\omega$. 
Assume that the conditions of the lemma are satisfied. We pick a particularly simple encoded effectivizer $F$, where, given codes for a set $X$ and a parameter $\vec{p}$, the set $\bar{X}:=\{x\in X:\phi(x,\vec{p})\}$ is encoded by taking the elements of $\text{tc}(\bar{X})$ in the same order in which they appear in the encoding of $X$. 
    We show by transfinite induction that, for any OTM-program $P$, any (coded) input set $x$ and any ordinal $\alpha$, the partial computation of $P^{F}(x)$ of length $\alpha$ is contained in HOD.     
    First, note that all such computations are clearly ordinal definable in the parameter $\alpha$ and thus elements of OD. 
    It thus suffices to see that the \textit{elements} of such a computation sequence belong to HOD. Any such element is an ordered pair $(\xi,(t,s,\rho))$, where $\xi$ is an ordinal less than or equal to $\alpha$ (indicating the point in time), $t$ is a map (or, if several tapes are used, a finite tuple of such) from some ordinal to $\{0,1\}$ (indicating the tape content at time $\xi$), $s$ is a natural number (indicating the inner state at time $\xi$) and $\rho$ is an ordinal (indicating the head position at time $\xi$). The only component that is not obviously an element of HOD is $t$. 
    If $\alpha$ is a limit ordinal, the claim is trivial, since, inductively, all elements of the computation sequence up to $\alpha$ will belong to HOD, as does the sequence itself, and, since HOD is a model of ZFC, hence so does the tape content obtained from $(t_{\iota}:\iota<\alpha)$ by the liminf-rule. If $\alpha=\beta+1$ is a successor ordinal, then, by induction, the computation sequence $((\iota,(t_{\iota},s_{\iota},\rho_{\iota})):\iota\leq\beta)$ belongs to HOD, and hence so does $t_{\beta}$. If the command executed at time $\beta$ is not the execution of the $F$-operator, then it is again trivial that $t_{\alpha}=t_{\beta+1}\in\text{HOD}$. This leaves us with the case that, at time $\beta$, $F(X,\lceil\phi\rceil,\vec{p})$ is executed for some set $X$, some formula $\phi$ and some parameter set $\vec{p}$. Since $X$ and $\vec{p}$ must then have codes $c_{X}$ and $c_{\vec{p}}$ written on the tapes at time $\beta$, these codes, and hence $X$ and $\vec{p}$ themselves, belong to HOD. It follows that $\bar{X}:=\{x\in X:\phi(x,\vec{p})\}\in\text{HOD}$. Now, by our choice of $F$, $F(X,\lceil\phi\rceil,\vec{p})$ is definable (without further parameters) from $c_{X}$ and $\bar{X}$. Thus $F(X,\lceil\phi\rceil,\vec{p})\in\text{HOD}$, completing the induction.
    
\end{proof}

Concerning full separation, we obtain the following:

\begin{theorem}{\label{full sep and hod}}
    (Fef) The class $S$ of sets that are OTM-computable relative to every effectivizer of Sep is HOD.
\end{theorem}
\begin{proof}
    By using the same argument as for Lemma \ref{sep and HOD} within (Fef), we see that $S\subseteq\text{HOD}$.

    Now let $x\in\text{HOD}$. Since $\text{HOD}\models\text{ZFC}$, and ZFC proves that every set has a code, HOD will also contain a set $c_{x}\subseteq\alpha\in\text{On}$ which codes $x$, for some $\alpha\in\text{On}$. By definition of HOD, there is an $\in$-formula $\phi$ and some $\rho\in\text{On}$ such that $c_{x}$ is the unique $y$ with $\phi(y,\rho)$. Let $F$ be an effectivizer for Sep. Now, using the parameters $\alpha$ and $\rho$, apply Sep to $\alpha$ with the formula $\psi(\iota):\Leftrightarrow\exists{y}(\phi(y,\rho)\wedge \iota\in y)$. The result will be $c_{x}$. Thus, a code for $x$ is OTM-computable relative to $F$.\footnote{We thank Philipp Schlicht for pointing out this simplification of our original argument (which is now used for Theorem \ref{pot sep hod}) to us.}
\end{proof}

\begin{remark}
    \begin{enumerate}
    \item In fact, the same argument, combined with the well-known fact that every ordinal definable set is already ordinal definable with a $\Sigma_{2}$-definition (see, e.g., \cite{Hamkins:MO}) yields that one can replace Sep with $\Sigma_{2}$-Sep in Theorem \ref{full sep and hod}.
    \item It is somewhat suggestive that, while plain OTM-computability with ordinal parameters yields all sets in $L$ (Koepke, \cite{Koepke}), OTM-computability relative to an effectivizer for separation yields HOD, which can be obtained by replacing first-order logic with second-order logic in the definition of $L$ (see \cite{KMV}). We do not know a direct connection between the two results, nor whether any other of the ``alternative $L$s'' described in \cite{KMV} can be given a computable characterization.
    \end{enumerate}
\end{remark}

We can now characterize precisely the models of ZFC which allow for a reduction of Pot to ($\Sigma_{2}$-)Sep.


\begin{theorem}{\label{pot sep hod}}
For all $n\geq 2$, $V=\text{HOD}$ is equivalent to Pot$\leq_{\text{OTM}}\Sigma_{n}$-Sep.
\end{theorem}
\begin{proof}
We assume that $n=2$; the other cases work entirely analogously.

First suppose that $\text{Pot}\leq_{\text{OTM}}\Sigma_{2}$-Sep. Now, $\Sigma_{2}$-Sep has a definable encoded effectivizer $\hat{F}$ as in the proof of Theorem \ref{full sep and hod}. It follows that, if $x$ is OTM-computable relative to $\hat{F}$, then so is $\mathfrak{P}(x)$. By Theorem \ref{full sep and hod}, it follows that, for every $x\in $\text{HOD}, we have $\mathfrak{P}(x)\in \text{HOD}$. By transitivity of $\text{HOD}$, this implies $V=\text{HOD}$.

On the other hand, suppose that $V=\text{HOD}$, and let $S$ be an encoded effectivizer for $\Sigma_{2}$-Sep. We show that, for all $x\in \text{HOD}$, $\mathfrak{P}^{\text{HOD}}(x)$ is uniformly $S$-OTM-computable in the input $x$. First, note that there is some ordinal $\gamma$ such that, for all $y\in\mathfrak{P}^{\text{HOD}}(x)$, $y$ is definable by a formula using only parameters $<\gamma$. Using $F$, such a $\gamma$ can in fact be computed on an OTM: Namely, let an OTM-program run through the ordinals, and, for each ordinal $\xi$ and each $\Sigma_{2}$-formula $\phi$, use $S$ to check whether the following statements holds true: 
 $$\forall{\rho\in\text{On}}\exists{\zeta<\xi}((\exists{!}z(z\subseteq x \wedge \phi(z,\rho)))\leftrightarrow(\exists{!}z(z\subseteq x \wedge \phi(z,\zeta))).$$
 Eventually, an ordinal $\xi$ will be found such that the answer is ``yes'' for all formulas $\phi$, which will be as desired. 
    
    Now, given $\gamma$, we can successively compute all elements of $\mathfrak{P}^{\text{HOD}}(x)$ as follows: Run through all pairs $(\phi,\zeta)$ of a formula $\phi$ and an ordinal $\zeta<\gamma$. Use $S$ to check the statement $\exists{!}z(z\subseteq x\wedge\phi(z,\zeta))$. For all pairs $(\phi,\zeta)$ for which this turns out to be true, use $F$ to compute $\{y\in x:\exists{z}(z\subseteq x\wedge \phi(z,\zeta)\wedge y\in z)\}$.
    After having run through all these pairs, compute a code for the set that has all the sets thus obtained as elements.
\end{proof}

We note several more reducibilities from and to Pot. 

\begin{theorem}{\label{pot}}
\begin{enumerate}
\item NextCard$\leq_{\text{OTM}}^{1,1}$Pot
\item If $V=L$, then NextCard$\equiv_{\text{OTM}}^{1,1}$Pot.
\item In general, Pot$\leq_{\text{OTM}}^{1,1}$NextCard is independent of ZFC.
\item Pot$\leq_{\text{OTM}}\Sigma_{3}$-Rep.\footnote{In fact, we have Pot$\leq_{\text{OTM}}\Pi_{2}$-Rep, for the obvious definition of the latter. Note, however, that this cannot be simplified further, since we require our formulas to be in prenex normal form, and expressing Pot requires bounded quantifiers on top of the initial unbounded universal quantifier.}
  (Fef) Moreover, we have $\phi\leq_{\text{OTM}}$Rep for every true $\in$-formula $\phi:\equiv\forall{x_{0}}\exists{y_{0}}...\forall{x_{n}}\exists{y_{n}}\psi(x_{0},y_{0},...,x_{n},y_{n}) $ which satisfies $$\forall{x_{0}}\exists{!}{y_{0}}...\forall{x_{n}}\exists!{y_{n}}\psi(x_{0},y_{0},...,x_{n},y_{n})$$ (i.e, for which witnesses are unique). 
\item NextCard$\leq_{\text{OTM}}^{1,1}\Sigma_{2}$-Sep.
\item For all $n\in\omega$, 
Pot$\leq_{\text{OTM}}^{1,1}\Sigma_{n}$-Sep is independent of ZFC. (Fef) Pot$\leq_{\text{OTM}}^{1,1}$Sep is independent of ZFC.
\end{enumerate}
\end{theorem}
\begin{proof}
\begin{enumerate}
    \item Let $F$ be an effectivizer for Pot, and let $\alpha$ be an ordinal. Then we obtain the cardinal successor of $\alpha$ from $F$ as follows: We have $F((\alpha))=\mathfrak{P}(\alpha)$. An OTM can check whether a set of ordinals codes a well-ordering.\footnote{See, e.g., \cite{CarlBuch}, Theorem 2.3.25; the argument is essentially the same given by Hamkins and Lewis for subsets of $\omega$, see \cite{Hamkins-Lewis}, Theorem 2.2.} We use this to compute the set of all codes for well-orderings in $\mathfrak{P}(\alpha)$ -- which will contain codes for all ordinals below $\alpha^{+}$, and from this set, compute a code for their sum, which will be $\alpha^{+}$. 
    \item One direction follows from (1). We consider the other direction: In $L$, we have $\mathfrak{P}(\alpha)=\mathfrak{P}^{L[\alpha^{+}]}(\alpha)$. Hence, we can compute $(\alpha^{+})^{L}$ using an effectivizer for NextCard, then use that to compute a code for $L_{(\alpha^{+})^{L}}$ and then extract all subsets of $\alpha$ contained in it.
    \item By (2), Pot$\leq_{\text{OTM}}^{1,1}$NextCard is consistent relative to ZFC. 

    For the other direction, let $L[g]$ be a generic extension of $L$ by a Cohen real $g\subseteq\omega$. Suppose for a contradiction that $P$ is an OTM-program that, in the parameter $\rho\in\text{On}$, reduces Pot to NextCard in $L[g]$. Let $F$ be an effectivizer of NextCard in $L[g]$. Then $F$ is definable in $L[g]$, and, since $L[g]$ has the same cardinals as $L$, $F$ is also definable in $L$. By absoluteness of computations, the computation of $P^{F}(\omega,\rho)$ is the same in $L$ and in $L[g]$. It follows that $g\in\mathfrak{P}^{L[g]}(\omega)\in L$, a contradiction.

    \item Let $F_{\text{Rep}}$ be an effectivizer for the replacement scheme. Let $\phi$ be a true formula with the given properties. 
    Now let (codes for) sets $a_{0},a_{1},...,a_{n}$ be given; we need to compute (successively) $b_{0},b_{1},...,b_{n}$ such that $\psi(a_{0},b_{0},...,a_{n},b_{n})$, where the computation of $b_{i}$ can only use $a_{0},...,a_{i}$ as its input. Let $k\leq n$, and suppose that the $b_{i}$ with $i<k$ have already been computed. Consider the formula $\phi^{k}:\equiv\forall{x_{k}}\exists{y_{k}}...\forall{x_{n}}\exists{y_{n}}\psi(a_{0},b_{0},...,a_{k-1},b_{k-1},x_{k},y_{k},...,x_{n},y_{n})$. 
    Then applying $F_{\text{Rep}}$ to $\phi_{k}$ and the set $\{a_{k}\}$ will produce a (code for) a set $Y\ni b_{k}$. We have already seen that separation is reducible to replacement and that a truth predicate is reducible to separation. Thus, given access to $F_{\text{Rep}}$, we can now run through $Y$ and check for each element $y\in Y$ whether it satisfies $\forall{x_{k+1}}\exists{y_{k+1}}...\forall{x_{n}}\exists{y_{n}}\psi(a_{0},b_{0},...,a_{k},y,x_{k+1},y_{k+1},...,x_{n},y_{n})$. By assumption, $Y$ contains a unique set $y$ with this property, which will eventually be found, after which we write it to the output tape. 
    
    This reduces $\phi$ to Rep.

    This argument then applies to Pot, noting that $y=\mathfrak{P}(x)$ can be expressed by a $\Pi_{2}$-formula in prenex normal form, and power sets are unique.
    \item Using the fact that an effectivizer for $\Sigma_{n}$-Sep allows us to evaluate the truth of $\Sigma_{n}$-statements, we can, for a given $\alpha\in\text{On}$, successively count through the ordinals, checking for each ordinal $\alpha$ whether $\alpha>\kappa\wedge\text{Card}(\alpha)$. The first ordinal for which this holds true is NextCard$(\alpha$).
    \item First suppose that $V=L$. 
    We know from Proposition \ref{card and hp} that NextCard$\leq_{\text{OTM}}\Sigma_{1}-\text{Sep}$; thus, using an effectivizer for $\Sigma_{1}-\text{Sep}$, we can compute the cardinal successor $\kappa$ of $\alpha$ in $L$.  
    Then compute $\mathfrak{P}^{L_{\kappa}}(x)$, which will be $\mathfrak{P}^{L}(x)$. 

    On the other hand, let $V$ be such that $V\neq\text{HOD}$.\footnote{For example, one can pick $V:=L[x]$, where $x$ is Cohen-generic over $L$. Since Cohen forcing is weakly homogenous (and definable without parameters in $L$), we then have $\text{HOD}^{L[x]}\neq$$\text{HOD}^{L}=L$ (see, e.g., \cite{WDR}, p. 2).} Then there must be some set $x\in\text{HOD}$ such that $\mathfrak{P}(x)\neq\mathfrak{P}^{\text{HOD}}(x)$ (for example, there must be some minimal ordinal $\alpha$ such that $V_{\alpha}\notin\text{HOD}$). Now, if $P$ was a program reducing Pot to $\Sigma_{n}$-Sep for some $n\in\omega$, then, for any effectivizer $F$ of $\Sigma_{n}$-Sep, we would have the result of the computation $P^{F}(x)\in\text{HOD}$ by Lemma \ref{sep and HOD}. But then, we would have $\mathfrak{P}(x)\in\text{HOD}$, a contradiction. Thus, it is consistent with ZFC that Pot$\nleq_{\text{OTM}}^{1,1}\Sigma_{n}$-Sep, for all $n\in\omega$.

\end{enumerate}    
\end{proof}

By a slight variant of the argument concerning Pot and NextCard, we see the stronger result that the reduction of Pot to Card is also independent of ZFC:

\begin{corollary}{\label{pot and card}}
       Pot$\leq_{\text{OTM}}$Card is independent of ZFC.
\end{corollary}
\begin{proof}
     In $L$, we have Pot$\leq_{\text{OTM}}$Card, as we can, for a given set $x$, use Card to determine first the cardinality $\kappa$ of $\text{tc}(x)$ and then (by searching through the ordinals) its successor $\kappa^{+}$ (this shows that NextCard$\leq_{\text{OTM}}$Card), from which we can compute a code for $L_{\kappa^{+}}$. Then the subsets of $x$ contained in $L_{\kappa^{+}}$ will be all subsets of $x$ in $L$, which we return as the powerset of $x$.

    On the other hand, let $g\subseteq\omega$ be Cohen-generic over $L$; then all constructible sets have the same cardinality in $L$ and in $L[g]$. Pick an effectivizer $F$ of Card in $L[g]$; we may assume without loss of generality that the restriction of $F$ to $L$ is definable in $L$ (e.g., by re-defining $F$ to pick, for each constructible set $a$, the $<_{L}$-smallest bijection between $a$ and its $L$-cardinality, and for all other sets the one that $F$ delivers). Now, if $(P,\rho)$ would witness the reduction in $L[g]$, then, in particular $P^{F}(\omega,\rho)$ would compute $\mathfrak{P}^{L[g]}(\omega)$, which contains the non-constructible element $g$ and is thus itself not constructible. However, by assumption on $F$, the whole computation will take place in $L$ and can thus not lead to a non-constructible output. 

\end{proof}

\begin{lemma}{\label{sep and rep}}
\begin{enumerate}
    \item Sep$\leq_{\text{OTM}}^{1,1}$Rep\footnote{This result was briefly mentioned in \cite{Ca2018}, section $6$.}
    \item (Fef) Rep$\leq_{\text{OTM}}^{1,1}$Sep$\oplus$Pot
    \item (Fef) Rep$\equiv_{\text{OTM}}^{1,1}$Coll; moreover, $\Sigma_{n}$-Rep$\leq_{\text{OTM}}\Sigma_{n}$-Coll, for all $n\in\omega$.
    \item (Fef) Coll$\leq_{\text{OTM}}^{1,1}$Sep$\oplus$Pot
    \item $\Sigma_{n}$-Coll$\leq_{\text{OTM}}^{1,1}\Sigma_{n}$-Sep$\oplus$Pot, for all $n\in\omega$.
    \item If $V=L$, then $\Sigma_{n}$-Rep$\equiv_{\text{OTM}}\Sigma_{n}$-Sep, for all $n\in\omega$. (Fef) Thus, also Rep$\equiv_{\text{OTM}}$Sep. 
    \item $\Sigma_{3}$-Rep$\leq_{\text{OTM}}^{1,1}$Sep is independent of ZFC.
\end{enumerate}
\end{lemma}
\begin{proof}
\begin{enumerate}
\item Given an $\in$-formula $\phi(x,\vec{z})$, a (coded) parameter $\vec{p}$, a (code for a) set $X$, and an effectivizer $F$ for the replacement scheme, let $\psi(x,y,\vec{z})$ be the $\in$-formula $(\phi(x,\vec{z})\wedge y=x)\vee(\neg\phi(x,\vec{z})\wedge y=X)$, use $F$ to obtain $Y^{\prime}:=F(\lceil\psi\rceil,\vec{z},X)$ and then compute $Y$ from $Y^{\prime}$ by deleting $X$. Since OTMs can test whether two sets of ordinals code the same set, the final step can be done on an OTM.
\item Let $F_{\text{Sep}}$ and $F_{\text{Pot}}$ be effectivizers for the separation scheme and the power set axiom, respectively, let $X$ be a set and let $\phi(x,y,\vec{z})$ be an $\in$-formula and $\vec{p}$ a parameter such that $\phi$ is functional on $X$ (i.e., such that $\forall{x\in X}\exists!{y}\phi(x,y,\vec{z})$) for $\vec{z}=\vec{p}$. By Proposition \ref{truth and sep}, we can obtain a truth predicate from $F_{\text{Sep}}$. Now use $F_{\text{Pot}}$ iteratively, starting with $\emptyset$, applying $F_{\text{Pot}}$ at sucessor levels and forming unions at limit levels (forming unions is computable on an OTM) to successively obtain the stages $V_{\iota}$ of the von Neumann hierarchy.\footnote{OTMs with access to a power set operator (``Power-OTMs'') are equivalent to the (full) Set Register Machines introduced and studied by Passmann, \cite{Passmann}; their ability to enumerate the $V$-hierarchy is essentially \cite{Passmann}, Proposition 3.10.} Use the truth predicate to test, for each of these levels, whether it satisfies $\forall{x\in X}(\exists{y}\phi(x,y,\vec{p})\rightarrow\exists{y\in V_{\iota}}\phi(x,y,\vec{p}))$. By assumption, such a level will eventually be found. This level is then our desired set.\footnote{One can now additionally use $F_{\text{Sep}}(\lceil\exists{x}\phi(x,y,\vec{p}\rceil,V_{\iota})$ to obtain the subset that \textit{only} contains sets required to exist by the respective instance of replacement.}
\item For all $n\in\omega$, $\Sigma_{n}$-Rep$\leq_{\text{OTM}}\Sigma_{n}$-Coll, and hence also Rep$\leq_{\text{OTM}}$Coll, are trivial, as every instance of replacement is also an instance of collection. 

For the other direction, we use an effective version of the usual argument that replacement implies collection in ZF: Let an $\in$-formula $\phi$, a parameter $\vec{p}$ and a set $X$ be given such that $\forall{x\in X}\exists{y}\phi(x,y,\vec{p})$. Consider the formula $\phi^{\prime}(x,\alpha,Y,\vec{p})$, which states that $\alpha\in\text{On}$ is minimal with the property that $\exists{y\in V_{\alpha}}\phi(x,y,\vec{p})$ and that $Y$ is the set of all $y\in V_{\alpha}$ that satisfy $\phi(x,y,\vec{p})$. Then $\phi^{\prime\prime}(x,Y,\vec{p}):\Leftrightarrow\exists{\alpha}\phi^{\prime}(x,\alpha,Y,\vec{p})$ defines a function from $X$ to $V$. Applying an effectivizer for $F_{\text{Rep}}$ to the set $X$ for $\phi^{\prime\prime}$ in the parameter $\vec{p}$, we obtain a set $Z$ such that $\bigcup{Z}$ will be the set required to exist by the given instance of Coll.

\item This follows immediately from (2), (3) and the transitivity of $\leq_{\text{OTM}}^{1,1}$. 

\item 
We use a similar idea as in (2), using an additional idea, which is basically an effective version of Scott's trick: Given a $\Sigma_{n}$-formula $\phi$, a parameter $\vec{p}$ and a set $X$. For each $x\in X$, we again use $F_{\text{Pot}}$ to enumerate the universe until we arrive at some $V_{\alpha}$ such that $\exists{y\in V_{\alpha}}\phi(x,y,\vec{p})$ (which can be identified by searching through each occurring $V$-level and testing each element using $F_{\Sigma_{n}-\text{Sep}}$). Now, again using $F_{\Sigma_{n}-\text{Sep}}$, compute the set $S_{x}:=\{y\in V_{\alpha}:\phi(x,y,\vec{p}\}$ and store the result. Once this has been done for all $x\in X$, compute $\bigcup_{x\in X}S_{x}$, which will be as desired.

\item Let $n\in\omega$, $\vec{p}$ a parameter, and let $\phi(x,y,z)$ be a $\Sigma_{n}$-formula which is functional for $z=\vec{p}$. Moreover, let an effectivizer $F$ for $\Sigma_{n}$-Sep be given, along with a (code for a) set $X$. 
Recall that, using $F$, we can evaluate the truth of $\Sigma_{n}$-formulas. The required OTM-program now works by enumerating $L$, and, for each $Y\in L$, performing the following:
\begin{enumerate}
    \item Running through $X$, and, for each $x\in X$, running through $Y$ and searching for some $y\in Y$ such that $\phi(x,y,\vec{p})$. If, for some $x\in X$, no such $y$ is found, the search through $L$ is continued.
    \item Running through $Y$, and, for each $y\in Y$, running through $X$ and searching for some $x\in X$ such that $\phi(x,y,\vec{p})$. If, for some $y\in Y$, no such $x$ is found, the search through $L$ is continued.
\end{enumerate}
Since $L$ is a model of replacement, the search will eventually reveal a set $Y$ for which both checks are successful; this will then be returned as the output. 
Since the procedure works uniformly in $\phi$, it also witnesses Rep$\leq_{\text{OTM}}^{1,1}$Sep in general, provided the relevant computations are defined. 
\item The relative consistency follows from (3). For the converse, we have already seen that Pot$\leq_{\text{OTM}}^{1,1}\Sigma_{3}$-Rep. If additionally $\Sigma_{3}$-Rep$\leq_{\text{OTM}}^{1,1}\Sigma_{n}$-Sep for any $n\in\omega$, then it follows by transitivity of $\leq_{\text{OTM}}^{1,1}$ that Pot$\leq_{\text{OTM}}^{1,1}\Sigma_{n}$-Sep. Thus, any model of ZFC where the latter fails is also one where $\Sigma_{3}$-Rep$\nleq_{\text{OTM}}^{1,1}\Sigma_{n}$-Sep, for all $n\in\omega$. But the existence of such models was already seen in the proof of Theorem \ref{pot}(6).
\end{enumerate}
\end{proof}


Given that Rep$\leq_{\text{OTM}}^{1,1}$Sep$\oplus$Pot, one may asked how far Pot can be weakened in this statement. We note that NextCard is not enough:

\begin{lemma}{\label{rep sep pot}}
It is consistent with ZFC that:

\begin{enumerate}
    \item Pot$\nleq_{\text{OTM}}^{1,1}$Sep$\oplus$NextCard
    \item $\Sigma_{3}$-Rep$\nleq_{\text{OTM}}^{1,1}$Sep$\oplus$NextCard
\end{enumerate}
\end{lemma}
\begin{proof}
Let $V=L[g]$, where $g\subseteq\omega$ is Cohen-generic over $L$. Then Card$=$Card$^{L}$, and so NextCard$=$NextCard$^{L}$. As in Theorem \ref{full sep and hod} above, it follows that every set computable by an OTM relative to every Sep$\oplus$NextCard-effectivizer is an element of $\text{HOD}^{L[g]}=L$. 
However, we have $x\notin L$, but $x\in\mathfrak{P}(\omega)$, so $\mathfrak{P}(\omega)$ is not OTM-computable from Sep$\oplus$NextCard. Thus Pot$\nleq_{\text{OTM}}^{1,1}$Sep$\oplus$NextCard, for, if $P$ and $\rho$ were a program and a parameter witnessing the reduction, then we could compute $\mathfrak{P}(\omega)$ from Sep$\oplus$NextCard by first computing $\omega$ and then applying Pot to it. 
Since Pot$\leq_{\text{OTM}}^{1,1}\Sigma_{3}$-Rep, it also follows that $\Sigma_{3}$-Rep$\nleq_{\text{OTM}}^{1,1}$Sep$\oplus$NextCard in $L[g]$.
\end{proof}

\subsection{The axiom of choice}

Finally, we have a look at the axiom of choice. The mutual effective reducibility of various versions of the axiom of choice was already considered in \cite{Ca2018}, where it was shown (ibid., Proposition $14$) that the usual equivalent formulations of this axiom are also mutually OTM-reducible to each other. One might to think that AC$\leq_{\text{OTM}}$Pot for the following reason: If $X$ is a set of non-empty sets, then $\mathfrak{P}(X\times\bigcup{X})$ will contain a choice function for $X$. Thus, an OTM with access to the power set operator can, uniformly in $X$, compute a set $S$ containing a choice function for $X$; and then one can search through $S$, identify a choice function and write it to the output.
Unfortunately, $S$ will, except in trivial cases, contain not only one, but many choice functions, and choosing one of these is as hard choosing from the elements of $X$ in the first place. 



\begin{defini}{\label{definable effectivizer}}
A definable effectivizer for an $\in$-formula $\phi$ is an effectivizer for $\phi$ that is definable as a class function in the language of ZFC (i.e., using the binary relation symbol $\in$ alone) and such that ZFC proves that $F$ is in fact an effectivizer for $\phi$.
\end{defini}

\begin{lemma}{\label{definability and reducibility}}
Let $\phi$ and $\psi$ be $\in$-formulas. Suppose that $\psi$ has a definable effectivizer, and that $\phi$ is ZFC-provably OTM-reducible to $\psi$. Then $\phi$ has a definable effectivizer.
\end{lemma}
\begin{proof}
    (Note that the statement is not immediate, since the program that implements the reduction requires an encoded effectivizer for $\psi$, rather than merely an effectivizer. In the absence of a definable global well-ordering, this may fail.)

    Suppose that $P$ ZFC-provably reduces $\phi$ to $\psi$, and let $F$ be a definable effectivizer for $\phi$, and let $M$ be a model of ZFC. Over $M$, consider the class forcing $(\mathbb{P},\preceq)$ for obtaining a global well-ordering of $M$.\footnote{That is, $\mathbb{P}$ consists of all well-orderings in $M$ and $\preceq$ is the end-extension relation. See, e.g., \cite{Gitman}.} Let $M[G]$ be the extension. Then $M[G]$ is a model of ZFC. Moreover, since this class forcing does not add any new sets (see \cite{HKS}, Example 3.2 and Lemma 3.5),\footnote{We thank Philipp Schlicht for pointing us to this reference.} the definition of $F$ yields the same function in $M[G]$ as in $M$. However, in $M[G]$, we can use the global well-ordering to encode $F$ as $\hat{F}$. Now, for any $x\in M$ and any code $c_{x}$ for $x$, $P^{\hat{F}}(c_{x})$ will by assumption terminate and output a code $c_{y}$ for some set $y$ such that $y$ depends only on $x$ (and not on $c_{x}$) and such that the function mapping $x$ to $y$ is an effectivizer for $\phi$. Since the computation $P^{\hat{F}}(c_{x})$ halts, it can only use set-many instances of $\hat{F}$. Thus, there is some $\hat{f}\in M$ such that $P^{\hat{f}}(c_{x})$ halts without requiring instances of $\hat{F}$ other than those encoded in $\hat{f}$ with output $c_{y}$. 
    
    We can thus define (in $M$) an effectivizer $F_{\phi}$ for $\phi$ as follows: $F_{\phi}(x)=y$ if and only if, for every code $c_{x}$ for $x$ and every function $\hat{f}$ which satisfies the definition of an encoded effectivizer for $\psi$ for all elements of its domain, $P^{\hat{f}}(c_{x})$ halts without requiring values of the oracle function not encoded in $\hat{f}$ and its output is a code for $y$.
    
\end{proof}

This suffices to clarify the ZFC-provable OTM-reduction relations between AC and the other non-effective axioms of ZFC:

\begin{theorem}{\label{choice and friends}}
\begin{enumerate}
    \item AC is not ZFC-provably OTM-reducible to $\Sigma_{n}$-Sep, for all $n\in\omega$. 
    \item AC is not ZFC-provably OTM-reducible to $\Sigma_{n}$-Rep, for all $n\in\omega$. 
    \item AC is not ZFC-provably OTM-reducible to Pot. 
\end{enumerate}
\end{theorem}
\begin{proof}
Since Pot, $\Sigma_{n}$-Sep and $\Sigma_{n}$-Rep all have definable effectivizers, AC being ZFC-provably OTM-reducible to any of these would, by Lemma \ref{definability and reducibility}, imply the existence of a definable global choice function. Thus, any transitive class model $M$ of ZFC that does not have a definable global choice function\footnote{By \cite{Hamkins:MO}, this is the case in any model of ZFC+V$\neq\text{HOD}$, such as a Cohen extension of $L$.} is a counterexample to the reduction.
\end{proof}

\begin{remark}
    We remark, however, that AC$\leq_{\text{OTM}}\phi$ is consistent with ZFC for every formula $\phi$, since, if $V=L$, then AC is in fact OTM-effective (\cite{Ca2018}, Proposition $3$) and thus trivially reducible to any formula. 
\end{remark}

\begin{proposition}{\label{reduction to AC}}
None of Pot, $\Sigma_{1}$-Sep and $\Sigma_{1}$-Rep is OTM-reducible to AC.
\end{proposition}
\begin{proof}
Pot$\nleq_{\text{OTM}}$AC follows from \cite{Ca2018}, Lemma $7$, combined with the fact that effectivizers for Pot increase cardinalities of the transitive closure of sets, while effectivizers for AC do not.

If we had $\Sigma_{1}\text{-Sep}\leq_{\text{OTM}}$AC in any model $M$ of ZFC, we would, in particular, have it in $L^{M}$ (since $L^{M}$ is $\Sigma_{1}$-definable in $M$); however, this would in fact make $\Sigma_{1}$-Sep OTM-effective for $L^{M}$ (since $L^{M}$ has an OTM-computable global well-ordering). But this cannot be, for HP$\leq_{\text{OTM}}\Sigma_{1}$-Sep.

A fortiori, since $\Sigma_{1}$-Sep$\leq_{\text{OTM}}\Sigma_{1}$-Rep, we also have $\Sigma_{1}$-Rep$\nleq_{\text{OTM}}$AC.
\end{proof}

\section{Conclusion and further work}



As a consequence of the results in this paper, we thus obtain the following overall picture: The set-theoretical axioms usually regarded as ``non-constructive'' or ``impredicative'' -- i.e., the power set axiom, the axiom of separation, the axiom of replacement and the axiom of choice -- are indeed all non-effective in the sense of this paper.\footnote{In \cite{CGP}, it was also noted that they are not OTM-realizable.} Concerning reducibilities, Rep allows one to make Pot and Sep effective, while Pot, Sep and AC are non-constructive in ``orthogonal'' ways. Moreover, many of the considered principles are reducible to each other in ``constructive'' universes, such as $L$ (or models of $V=\text{HOD}$), but not in general. This suggests studying more precisely the conditions under which certain reductions are possible. 

The above investigations can, of course, be extended to further set-theoretical statements. 

We may also ask about the overall structure of the ordering induced on the set of set-theoretical formulas by $\leq_{\text{OTM}}$. Is it linear? Is it dense? Which elements $e$ (if any) can be ``decomposed'' in the sense that $e\equiv_{\text{OTM}}a\oplus b$ where $a,b<_{\text{OTM}}e$?

The statement HP can be regarded as an analogue of the Turing jump in the context of OTM-effective reducibility. In general, degree-theoretical questions can be asked about set-theoretical statements. We regard the following one as particularly interesting:

\begin{question}
    Is there a ``natural'' set-theoretical statement $\phi$ such that $\top<_{\text{OTM}}\phi<_{\text{OTM}}HP$?
\end{question}

As already mentioned above, the notion of OTM-reducibility can also be applied to formulas with infinitely many quantifier alternations. In particular, one can define an effectivizer for the axiom of determinacy, AD, as a map 
$F:\mathfrak{P}(\omega^{\omega})\times \omega^{<\omega}\rightarrow\omega$ 
such that, for all $X\subseteq \omega^{\omega}$, $s\mapsto F(X,s)$ is a winning strategy for $X$. 

\begin{question}
    Which of the ZF-provable consequences of AD are OTM-reducible to AD?
\end{question}

Frequently, non-effectivity or non-reducibility results not from the inability of an OTM to compute sets satisfying a certain property, but to pick out one of them. Thus, it should be worthwhile to study a notion of effectivity and reducibility based on the ability to compute a set of witnesses, rather than a single one; this would correspond to the notion of algebraic definability in set theory, see Hamkins and Leahy, \cite{Hamkins-Leahy}.



\section{Acknowledgements}

This paper is an expanded version of our CiE paper \cite{Ca2025}. It has profited considerably from the comments and corrections of the five (!) anonymous referees who reviewed that paper. 
We also thank the organizers of the workshop ``Weihrauch Complexity: Structuring the Realm of Non-Computability'' in Dagstuhl (Dagstuhl seminar 25131) and Schloss Dagstuhl for its hospitality; several of the results in this paper that go beyond those in \cite{Ca2025} were obtained during that workshop.

\end{document}